\documentclass[11pt,reqno,tbtags,a4paper]{amsart}
\usepackage{amssymb}
\usepackage{xpunctuate}
\usepackage{url}
\usepackage[square,numbers]{natbib}
\bibpunct[, ]{[}{]}{;}{n}{,}{,}

\title
{Real trees}

\date{8 June, 2020; revised 14 March 2023}

\author{Svante Janson}
\thanks{Partly supported by the Knut and Alice Wallenberg Foundation}
\address{Department of Mathematics, Uppsala University, PO Box 480,
SE-751~06 Uppsala, Sweden}
\email{svante.janson@math.uu.se}
\newcommand\urladdrx[1]{{\urladdr{\def~{{\tiny$\sim$}}#1}}}
\urladdrx{http://www2.math.uu.se/~svante/}

\subjclass[2010]{} 

\numberwithin{equation}{section}

\renewcommand\le{\leqslant}
\renewcommand\ge{\geqslant}

\allowdisplaybreaks

\theoremstyle{plain}
\newtheorem{theorem}{Theorem}[section]
\newtheorem{lemma}[theorem]{Lemma}

\theoremstyle{definition}

\newtheorem{exampleqqq}[theorem]{Example}
\newenvironment{example}{\begin{exampleqqq}}
  {\hfill\qedsymbol\end{exampleqqq}}

\newtheorem{remarkqqq}[theorem]{Remark}
\newenvironment{remark}{\begin{remarkqqq}}
  {\hfill\qedsymbol\end{remarkqqq}}

\newtheorem{definition}[theorem]{Definition}

\theoremstyle{remark}

\newenvironment{romenumerate}[1][-10pt]{
\addtolength{\leftmargini}{#1}\begin{enumerate}
 \renewcommand{\labelenumi}{\textup{(\roman{enumi})}}%
 }{\end{enumerate}}

\newenvironment{PXenumerate}[1]{
\begin{enumerate}
 \renewcommand{\labelenumi}{\textup{(#1\arabic{enumi})}}%
 }{\end{enumerate}}

\newenvironment{PXaenumerate}[1]{
\begin{enumerate}
 \renewcommand{\labelenumi}{\textup{(#1\alph{enumi})}}%
 }{\end{enumerate}}

\newcounter{oldenumi}
{\setcounter{oldenumi}{\value{enumi}}
\begin{romenumerate} \setcounter{enumi}{\value{oldenumi}}}
{\end{romenumerate}}

\newcounter{thmenumerate}

\newcounter{xenumerate}   

\newcommand\pfitemx[1]{\par#1:}
\newcommand\pfitemref[1]{\pfitemx{\ref{#1}}}

\newcommand{\refT}[1]{Theorem~\ref{#1}}

\newcommand{\refL}[1]{Lemma~\ref{#1}}
\newcommand{\refLs}[1]{Lemmas~\ref{#1}}
\newcommand{\refR}[1]{Remark~\ref{#1}}

\newcommand{\refS}[1]{Section~\ref{#1}}

\newcommand{\refD}[1]{Definition~\ref{#1}}
\newcommand{\refE}[1]{Example~\ref{#1}}
\newcommand{\refEs}[1]{Examples~\ref{#1}}

\newcommand\REM[1]{{\raggedright\texttt{[#1]}\par\marginal{XXX}}}
\newcommand\XREM[1]{\relax}

\begingroup
  \divide\count255 by 60
  \multiply\count255 by -60
  \advance\count255 by \time
  \ifnum \count255 < 10 \xdef\klockan{\the\count1.0\the\count255}
  \else\xdef\klockan{\the\count1.\the\count255}\fi
\endgroup

\newcommand\nopf{\qed}   

\newcommand\set[1]{\ensuremath{\{#1\}}}
\newcommand\bigset[1]{\ensuremath{\bigl\{#1\bigr\}}}

\newcommand\xpar[1]{(#1)}
\newcommand\bigpar[1]{\bigl(#1\bigr)}

\def\rompar(#1){\textup(#1\textup)}    

\def\xexp(#1){e^{#1}}

\newcommand\ntoo{\ensuremath{{n\to\infty}}}

\newcommand\ktoo{\ensuremath{{k\to\infty}}}

\newcommand\bmin{\wedge}
\newcommand\bmax{\vee}

\newcommand\downto{\searrow}
\newcommand\upto{\nearrow}
\newcommand\punkt{\xperiod}    
    
\newcommand\ie{i.e\punkt}
\newcommand\eg{e.g\punkt}

\newcommand\cf{cf\punkt}

\newcommand\bbR{\mathbb R}

\newcommand\bbT{\mathbb T}

\newcounter{CC}
\newcounter{cc}

\newcommand\ga{\alpha}

\newcommand\gd{\delta}
\newcommand\gD{\Delta}
\newcommand\gf{\varphi}
\newcommand\gam{\gamma}

\newcommand\gl{\lambda}

\newcommand\gs{\sigma}

\newcommand\eps{\varepsilon}

\renewcommand\phi{\xxx}  

\newcommand\cA{\mathcal A}
\newcommand\cB{\mathcal B}

\newcommand\cD{\mathcal D}

\newcommand\cH{\mathcal H}

\newcommand\tT{{\tilde T}}

\newcommand\qw{^{-1}}

\newcommand\qq{^{1/2}}

\newcommand\oi{\ensuremath{[0,1]}}

\newcommand\ooio{(0,1)}
\newcommand\ooo{[0,\infty)}

\newcommand\setoi{\set{0,1}}

\newcommand\lhs{left-hand side}

\newcommand\xoo{_1^\infty}

\newcommand\dxx[2]{d_{#1,#2}}
\newcommand\dxy{\dxx xy}
\newcommand\dxz{\dxx xz}
\newcommand\dxw{\dxx xw}
\newcommand\dyz{\dxx yz}
\newcommand\oxy{[0,\dxy]}
\newcommand\gfxy{\gf_{x,y}}
\newcommand\refTT{\ref{T2first}--\ref{T2last}}
\newcommand\Tq[1]{T\setminus\set{#1}}
\newcommand\Tz{\Tq{z}}
\renewcommand\deg[1]{\gd(#1)}
\newcommand\degg[2]{\gd_{#1}(#2)}
\newcommand\hzeta{h}
\newcommand\ddd{\gD}
\newcommand\hT{\widehat T}
\newcommand\FPC{four-point inequality}
\newcommand\skel{^o}
\newcommand\leaves{^{\mathsf{L}}}
\newcommand\hmi{\cH^1}
\newcommand\hmx[1]{\cH^{#1}}
\newcommand\roott{\emptyset}
\newcommand\LL{L^+}
\newcommand\setoioo{\setoi^\infty}
\newcommand\gsf{$\gs$-field}
\newcommand\Tsub{S}
\newcommand\oell{[0,\ell]}
\newcommand\bex{\mathbf{e}}

\newcommand{\Levy}{L\'evy}

\begin{document}

\begin{abstract} 
We survey the definition and some elementary properties of real trees.
There are no new results, as far as we know.
One purpose is to give a number of different definitions and show the
equivalence between them. 
We discuss also, for example,
the four-point inequality, the length measure
and the connection to the theory of Gromov hyperbolic spaces.
Several examples are given.
\end{abstract}

\maketitle

\section{Introduction}\label{S:intro}

This is a survey of various equivalent definitions of real trees and some
properties of them, mostly with  proofs.
We do not think that there are any new results.
Most of the paper considers deterministic real trees, but we include also
some brief comments on random real trees.
For further results, see for example \cite{Dress}, \cite{Evans},
\cite{LeGall2006}, and the references there. 

\subsection{Definition}
There are several different but equivalent definitions of real trees
(also called {$\bbR$-trees}).
We collect several of them as follows. 
We define below conditions \ref{T1} and \refTT{} on
a metric space $(T,d)$; we will show that 
assuming \ref{T1}, the conditions \refTT{} are all equivalent.
We then make the following definition.  (Which we state already here,
although it is not yet justified.)

\begin{definition}
  A \emph{real tree} (or \emph{$\bbR$-tree})
is a non-empty metric space $T=(T,d)$ 
that satisfies condition \ref{T1} and one (and thus all) of \refTT.
\end{definition}

\begin{remark}
Some authors assume also that the metric space $T$ is complete.
We will not do so. See further \refR{Rcomplete} below.
Note also that in many applications, $T$ is assumed to be compact;
again we do not assume this.
\end{remark}

Another equivalent, and related, definition is given in 
\cite[Definition 3.15]{Evans}. 
A characterization of a different kind of real trees is given in \refT{T4}.

\subsection{Some notation}

Throughout, $T=(T,d)$ is a (non-empty) metric space.
We often write $\dxy$ for $d(x,y)$.

$B(x,r):=\set{y:d(x,y)<r}$ denotes the open ball with centre $x\in T$ and
radius $r>0$.

If $\psi_1:[0,a]\to T$ and $\psi_2:[0,b]\to T$ are continuous maps with
$\psi_1(a)=\psi_2(0)$, their \emph{concatenation}
$\psi_1*\psi_2:[0,a+b]\to T$ is  defined by
\begin{align}\label{concat}
  \psi_1*\psi_2(t):=
  \begin{cases}
    \psi_1(t),& 0\le t\le a,
\\
\psi_2(t-a),& a \le t \le a+b.
  \end{cases}
\end{align}
The concatenation is clearly a continuous map $[0,a+b]\to T$.

If $s,t\in\bbR$, then $s\bmin t:=\min\set{s,t}$
and $s\bmax t:=\max\set{s,t}$. (These operations have priority over addition
and subtraction.)

\section{The conditions}\label{Scond}

In this section we state the conditions on a metric space $T=(T,d)$, 
beginning with the central \ref{T1}.

\begin{PXenumerate}{T}
  
\item \label{T1}
For any $x,y\in T$, there exists a unique isometric embedding
$\gfxy$ of the closed interval $\oxy\subset\bbR$ into $T$ such that
$\gfxy(0)=x$ and $\gfxy(\dxy)=y$. 
\end{PXenumerate}

Assume that \ref{T1} holds.
We then
denote the image $\gf_{x,y}(\oxy)\subseteq T$ by $[x,y]$; thus $[x,y]$ is
a connected compact subset of $T$, homeomorphic with $\oi$ if $x\neq y$.
We similarly define $[x,y)=[x,y]\setminus\set y$,
$(x,y]=[x,y]\setminus\set x$,
$(x,y)=[x,y]\setminus\set{x,y}$.
(If $x=y$, then $[x,x]=\set x$, 
and $[x,x)=(x,x]=(x,x)=\emptyset$.)

Obviously, $\gf_{y,x}(t)=\gfxy\xpar{\dxy-t}$ and $[y,x]=[x,y]$.

Furthermore, still assuming \ref{T1}, let $x,y,z\in T$. Since $\gfxy$ and
$\gf_{x,z}$ are isometries,
\begin{align}\label{xyz}
  [x,y]\cap[x,z]
=\bigset{\gfxy(t):t\in[0,\dxy\bmin\dxz],\,\gfxy(t)=\gf_{x,z}(t)}.
\end{align}
We define, noting that the maximum exists (\ie, the supremum is attained)
by continuity,
\begin{align}\label{ddd}
  \ddd(x,y,z)&:=
\max\bigset{t\in[0,\dxy\bmin\dxz]:\gfxy(t)=\gf_{x,z}(t)},
\\\label{gam}
\gam(x,y,z)&:=\gfxy(\ddd(x,y,z))
=\gf_{x,z}(\ddd(x,y,z))
\in[x,y]\cap[x,z].
\end{align}

Further properties of these objects are given in \refS{ST1}.

We turn to the conditions \refTT. 
These are stated for a metric space $T=(T,d)$ such that \ref{T1} holds, so
we can use the notations just introduced.

\begin{PXaenumerate}{T2}
  
\item \label{T2first}\label{T2disco}
For any $x,y\in T$ and any $z\in(x,y)$, $x$ and $y$ are in
different components of $T\setminus\set z$. 

\item \label{T2Y}
For any $x,y,z\in T$,
$[y,z]\subseteq [x,y]\cup[x,z]$.

\item \label{T2tri}
For any $x,y,z\in T$,
$[x,y]\cap[x,z]\cap[y,z]\neq\emptyset$.

\item \label{T2ddd}
For any $x,y,z\in T$,
$\gam(x,y,z)\in [y,z]$.

\item \label{T2dist} 
For any injective continuous map $\psi:\oi\to T$, 
\begin{align}
  d\bigpar{\psi(0),\psi(t)} +   d\bigpar{\psi(t),\psi(1)}
=   d\bigpar{\psi(0),\psi(1)},
\qquad t\in\oi.
\end{align}

\item \label{T2inj} 
For any injective continuous map $\psi:\oi\to T$, 
$\psi(\oi)\subseteq [\psi(0),\psi(1)]$.

\item \label{T2inj=} 
For any injective continuous map $\psi:\oi\to T$, 
$\psi(\oi)= [\psi(0),\psi(1)]$.

\item \label{T2=} 
Any injective continuous map $\psi:\oi\to T$ 
equals $\gfxy$ up to parametrization, 
where $x=\psi(0)$ and $y=\psi(1)$;
\ie,
$\psi=\gfxy\circ h$ for some strictly increasing homeomorphism $\oi\to\oxy$.

\item \label{T2injsurj}
For any injective continuous map $\psi:\oi\to T$, 
$\psi(\oi)\supseteq [\psi(0),\psi(1)]$.

\item \label{T2surj}
For any continuous map $\psi:\oi\to T$, 
$\psi(\oi)\supseteq [\psi(0),\psi(1)]$.

\label{T2last} 
\end{PXaenumerate}

As said in the introduction, we have the following equivalences.
\begin{theorem}\label{T=}
  Assume that $T=(T,d)$ is a metric space such that \ref{T1} holds.
Then
\refTT{} are all equivalent.
\end{theorem}

The proof is given in \refS{SpfT=}.

\begin{remark}\label{Rnot}
  Condition \ref{T1} alone is not sufficient. 
Examples of spaces satisfying \ref{T1} without being real trees
are the Euclidean space
$\bbR^d$, $d\ge2$, and any convex subset of $\bbR^d$ of dimension $\ge2$;
for example the unit disc.
\end{remark}

\section{Consequences of \ref{T1}}\label{ST1}

In this section we assume \ref{T1} (and sometimes further conditions), 
and show some lemmas used in the proof
of \refT{T=}.

\begin{lemma}\label{L0}
Suppose that \ref{T1} holds. Then
$T$ is connected, pathwise connected, and  locally pathwise connected.
Hence, if $V\subset T$ is an open subset of $T$, then $V$ is a union of
open (pathwise) connected components.
\end{lemma}

\begin{proof}
$T$ is obviously pathwise connected by \ref{T1}. Thus, $T$ is connected.

Furthermore,
$T$ is locally pathwise connected, since every open ball
$B(x,r)$  
is pathwise connected. 
(Every $y\in B(x,r)$ is connected to $x$ by the path $[x,y]\subseteq B(x,r)$.)
\end{proof}

In particular, if $z\in T$, then 
the components of $T\setminus\set z$ are open and pathwise connected.
These (path) components are called the \emph{branches} at $z$; see also
\refS{Sbranch}.

\begin{lemma}\label{L2}
  Suppose that \ref{T1} holds.
If $x,y\in T$ and $z,w\in[x,y]$, then $[z,w]\subseteq[x,y]$,
and, furthermore, 
\begin{align}\label{kia}
  \gf_{z,w}(t):=\gfxy(\dxz+t), \qquad 0\le t\le\dxx zw.
\end{align}

\end{lemma}
\begin{proof}
By symmetry, we may assume $\dxz\le  \dxw$.
Since $\gfxy$ is an isometry with $\gfxy(0)=x$, we  have $z=\gfxy(\dxz)$
and $w=\gfxy(\dxw)$; furthermore, $d(z,w)=|\dxz-\dxw|=\dxw-\dxz$.
Let 
\begin{align}
  \gf(t):=\gfxy(\dxz+t), \qquad 0\le t\le \dxw-\dxz=\dxx zw.
\end{align}
Then $\gf$ is an isometry and it follows that $\gf=\gf_{z,w}$. The result
follows. 
\end{proof}

\begin{lemma}
  \label{L1}
Suppose that \ref{T1} holds.
Then, for any $x,y\in T$,
\begin{align}
[  x,y] = \set{z:d(x,z)+d(z,y)=d(x,y)}.
\end{align}
\end{lemma}

\begin{proof}
  If $z\in[x,y]$, then by definition $z=\gfxy(s)$ for some $s\in\oxy$.
Since $\gfxy$ is an isometry, we have $d(x,z)=d\bigpar{\gfxy(0),\gfxy(s)}=s$ 
and $d(z,y)=d\bigpar{\gfxy(s),\gfxy(\dxy)}=d(x,y)-s$. Hence,
\begin{align}\label{ll1a}
  d(x,z)+d(z,y)=s + \bigpar{d(x,y)-s} = d(x,y).
\end{align}

Conversely, suppose that $z\in T$ with $d(x,z)+d(z,y)=d(x,y)$. 
Define 
$\gf:\oxy\to T$ as the concatenation $\gf:=\gf_{x,z}*\gf_{z,y}$,
see \eqref{concat}.
Then, 
$\gf$ is  a continuous map $\oxy\to T$ with $\gf(0)=x$ and
$\gf(\dxy)=y$. Furthermore, $\gf$ is an isometry on $[0,\dxz]$ and on
$[\dxz,\dxy]$. It follows that if 
$s\in [0,\dxz]$ and $t\in [\dxz,\dxy]$, then, 
by the triangle inequality,
\begin{align}\label{l1c}
    d\bigpar{\gf(s),\gf(t)}
&\le
  d\bigpar{\gf(s),\gf(\dxz)}
+   d\bigpar{\gf(\dxz),\gf(t)}
\notag\\&
=\bigpar{\dxz-s}+\bigpar{t-\dxz}
=t-s.
\end{align}
On the other hand, if we have strict inequality $d\bigpar{\gf(s),\gf(t)}<t-s$
for some $s,t$ with $0\le s\le t\le \dxy$, then, similarly,
\begin{align}\label{l1d}
  d(x,y) 
&\le d\bigpar{x,\gf(s)}+d\bigpar{\gf(s),\gf(t)}+d\bigpar{\gf(t),y}
\notag\\&
= s+d\bigpar{\gf(s),\gf(t)}+ \dxy-t
\notag\\&
< s + (t-s) + (\dxy-t) = \dxy=d(x,y),
\end{align}
a contradiction.

Consequently, $\gf$ is an isometry, and thus $\gf=\gfxy$, by the uniquness
assumption in \ref{T1}.
Hence, $z=\gf(\dxz)=\gfxy(\dxz)\in[x,y]$.
\end{proof}

\begin{lemma}\label{L22}
  Suppose that \ref{T1} holds.
If $x,y\in T$ and $z,w\in[x,y]$ with $z\in[x,w]$, then $w\in[z,y]$.
\end{lemma}

\begin{proof}
Since $z\in[x,w]$, we have $\dxz\le \dxw$. Hence, by \refL{L1},
\begin{align}
  d_{y,w} = \dxy-\dxw \le \dxy-\dxz=d_{y,z}.
\end{align}
Hence, by \refL{L2},
\begin{align}
  w = \gf_{y,x}(d_{y,w})
= \gf_{y,z}(d_{y,w})
\in [y,z].
\end{align}
\end{proof}

\begin{lemma}\label{LgD}
  Suppose that \ref{T1} holds.
Then, for any $x,y,z\in T$,
\begin{align}\label{lgda}
\bigset{t\in[0,\dxy\bmin\dxz]: \gfxy(t)=\gf_{x,z}(t)}=[0,\ddd(x,y,z)]
\end{align}
and
\begin{align}\label{lgdb}
  [x,y]\cap[x,z] = [x,\gam(x,y,z)]
=\bigset{\gfxy(t): t\in [0,\ddd(x,y,z)]}.
\end{align}
\end{lemma}

\begin{proof}
Let
\begin{align}\label{J}
J:=\bigset{t\in[0,\dxy\bmin\dxz]: \gfxy(t)=\gf_{x,z}(t)}.
\end{align}
Thus, the definition \eqref{ddd} says $\ddd(x,y,z)=\max J$.
If $t\in J$, let $w:=\gfxy(t)\in[x,y]\cap[x,z]$.
Then $[x,w]\subseteq [x,y]\cap [x,z]$ by \refL{L2}.
Hence,  if $0\le s\le t$, then
\begin{align}\label{loss}
  \gfxy(s)=\gf_{x,w}(s)=\gf_{x,z}(s)
\end{align}
and consequently, $s\in J$.
This shows that $J$ is an interval, and 
\eqref{lgda} follows.

Finally, \eqref{lgdb} follows from \eqref{lgda},
using \eqref{xyz}, \eqref{gam} and \eqref{loss}.
\end{proof}

\begin{lemma}
  \label{Lddd}
Suppose that \ref{T1} and \ref{T2ddd} hold.
Then, for any $x,y,z\in T$,
\begin{align}\label{gam3}
  [x,y]\cap[x,z]\cap[y,z]=\set{\gam(x,y,z)}
\end{align}
and
\begin{align}\label{ddd3}
  \ddd(x,y,z)=\tfrac12\bigpar{d(x,y)+d(x,z)-d(y,z)}.
\end{align}
In particular, 
$\ddd$ is a continuous function $T^3\to\bbR$,
and
$\gam(x,y,z)$ is a symmetric function of $x,y,z$.
\end{lemma}

\begin{proof}
  By the definition \eqref{gam} and the assumption \ref{T2ddd},
  \begin{align}\label{gambit}
    \gam(x,y,z)\in[x,y]\cap[x,z]\cap[y,z].
  \end{align}
Let $w\in[x,y]\cap[x,z]\cap[y,z]$.
Then, by \refL{L1},
\begin{align}
  \dxw+\dxx yw &= \dxy, \label{gamxy}\\
  \dxw+\dxx zw &= \dxz, \label{gamxz}\\
  \dxx yw+\dxx zw &= \dyz, \label{gamyz}
\end{align}
and consequently
\begin{align}\label{gamman}
 2 \dxw= \dxy+\dxz-\dyz.
\end{align}
Hence, $\dxw$ is uniquely determined by $x,y,z$, and thus so is
$w=\gfxy(\dxw)$. Consequently, \eqref{gambit} implies \eqref{gam3}.
Furthermore, \eqref{gam} implies
$\ddd(x,y,z)=d(x,\gam(x,y,z))$ and thus \eqref{ddd3} follows from
\eqref{gamman}.
\end{proof}

\begin{lemma}\label{L5}
  Suppose that \ref{T1} and \ref{T2ddd} hold,
and let $x,y,z\in T$.
\begin{romenumerate}
  
\item \label{L5a}
Then
\begin{align}\label{l5a}
  z\in[x,y] \iff \ddd(z,x,y)=0.
\end{align}

\item \label{L5b}
If $z\notin[x,y]$ and $d(w,z)<\ddd(z,x,y)$, then
\begin{align}\label{l5b}
  \ddd(x,y,w)=\ddd(x,y,z).
\end{align}
\end{romenumerate}
\end{lemma}

\begin{proof}
  \pfitemref{L5a}
Immediate by \refL{L1} and \eqref{ddd3}.

\pfitemref{L5b}
Let, 
using the symmetry of $\gam$ in \refL{LgD}, 
$u:=\gam(x,y,z)=\gam(z,x,y)$ and 
$v:=\gam(x,z,w)$. Then $v\in[z,w]$ and thus, using the assumption,
\begin{align}\label{l5c}
d(z,v)\le d(z,w) <\ddd(z,x,y)=d(z,u). 
\end{align}
Since $u,v\in[x,z]$, this implies
\begin{align}
  d(x,u)=d(x,z)-d(u,z) 
<d(x,z)-d(v,z)=d(x,v).
\end{align}
Hence, if $d(x,u)<t<d(x,v)$, then, by \eqref{xyz},
\begin{align}
  \gf_{x,w}(t)=\gf_{x,z}(t)\neq\gfxy(t)
\end{align}
On the other hand, $u\in[x,y]$ too, and thus
\begin{align}
  \gf_{x,w}(\dxx xu)=\gf_{x,z}(\dxx xu)=u=\gfxy(\dxx u).
\end{align}
Hence, $\ddd(x,y,w)=\dxx xu=\ddd(x,y,z)$ by \refL{LgD}.
\end{proof}

\section{Proof of \refT{T=}}\label{SpfT=}

\ref{T2disco}$\implies$\ref{T2surj}:
If \ref{T2disco} holds and $\psi:\oi\to T$ is a continuous map, let
$x:=\psi(0)$ and $y:=\psi(1)$.
If $z\in [x,y]$ but $z\notin \psi(\oi)$, then $\psi$ is a curve connecting
$x$ and $y$ in $T\setminus\set z$, which contradicts \ref{T2disco}.


\ref{T2surj}$\implies$\ref{T2Y}:
Define $\psi:=\gf_{y,x}*\gf_{x,z}$, see \eqref{concat}. This is a continuous map
$[0,d_{y,x}+d_{x,z}]\to T$ and thus $\psi_1(t):=\psi\bigpar{t(d_{y,x}+d_{x,z})}$
is a continuous map $\psi_1:\oi\to T$.
We have $\psi_1(\oi)=[y,x]\cup[x,z]$,  $\psi_1(0)=y$ and $\psi_1(1)=z$.
Hence \ref{T2surj} yields \ref{T2Y}.

\ref{T2Y}$\implies$\ref{T2tri}:
Let $A:=[y,z]\cap[x,y]$ and $B:=[y,z]\cap[x,z]$. These are closed subsets of
$[y,z]$ and $A\cup B=[y,z]$ by assumption; furthermore, $A$ and $B$ are both
nonempty since $y\in A$ and $z\in B$. Since $[y,z]$ is homeomorphic to
$\oi$, it is connected. Consequently, $A\cap B\neq\emptyset$.

\ref{T2tri}$\implies$\ref{T2ddd}:
By \eqref{gam}, it remains to show that $\gam(x,y,z)\in[y,z]$.

Let $w\in [x,y]\cap[x,z]\cap[y,z]$.
By \refL{LgD}, then $w\in [x,\gam(x,y,z)]$.
Since $w,\gam(x,y,z)\in[x,y]$, \refL{L22} shows that
$  \gam(x,y,z)\in [w,y]$.
Hence,
using also $w\in[y,z]$ and  \refL{L2},
\begin{align}
  \gam(x,y,z)\in [w,y]=[y,w]\subseteq[y,z].
\end{align}

\ref{T2ddd}$\implies$\ref{T2disco}:
Let $z\in(x,y)$. Then $z=\gfxy(\dxz)$ with $0<\dxz<\dxy$.
Partition $T\setminus\set{z}=U_1\cup U_2\cup U_3$ where
\begin{align}
  U_1&:=\set{w\in \Tz:\ddd(x,y,w)<\dxz}, \\
  U_2&:=\set{w\in \Tz:\ddd(x,y,w)=\dxz}, \\
  U_3&:=\set{w\in \Tz:\ddd(x,y,w)>\dxz}.
\end{align}
By \refL{Lddd}, $\ddd(x,y,w)$ is a continuous function of $w$, and thus
$U_1$ and $U_3$ are open subsets of $\Tz$ (and of $T$).
Furthermore, $U_2$ is open by \refL{L5}.

Hence, $\Tz=U_1\cup U_2\cup U_3$ is a partition into three disjoint open
sets. Each connected component of $\Tz$ has to be a subset of one of these,
and since $x\in U_1$ and $y\in U_3$, $x$ and $y$ are in different components.
(Although not needed, it is easy to see that $U_1$ and $U_3$ are connected,
while $U_2$ may be empty, connected or disconnected with any number of
components, finite or infinite.)

\ref{T2surj}$\implies$\ref{T2injsurj}: Trivial.

\ref{T2injsurj}$\implies$\ref{T2inj}:
Suppose that \ref{T2inj} fails, and let $\gf:\oi\to T$ be an injective 
continuous map
such that $\gf(u)\notin[x,y]$ for some $u\in(0,1)$, where $x=\gf(0)$, $y=\gf(1)$.
Since $[x,y]$ is a compact, and thus closed, subset of $T$, the set
$U:=\set{t\in\ooio:\gf(t)\notin[x,y]}$ is an open subset of $(0,1)$.
Hence, the component of $U$ containing $u$ is an open interval
$(a,b)\subseteq\ooio$. 

Let $z:=\gf(a)$ and $w:=\gf(b)$, and note that $z,w\in[x,y]$.  Since $a<b$
and $\gf$ is injective, we have $z\neq w$. 
By \refL{L2}, $[z,w]\subseteq [x,y]$.
The  map $\psi(t):=\gf(a+(b-a)t)$ 
is an injective continuous map $\oi\to T$  such that
$\psi(0)=z$, $\psi(1)=w$ and
\begin{align}
\psi(\oi) \cap [z,w]   
= \gf([a,b]) \cap [z,w]   
\subseteq
\gf([a,b]) \cap [x,y] 
=\gf(\set{a,b}) = \set{z,w} 
.\end{align}
Hence,
if $v\in(z,w)$, then $v\notin\psi(\oi)$.
Thus,
\ref{T2injsurj} does not hold.

\ref{T2inj}$\implies$\ref{T2ddd}:
Let $x,y,z\in T$, and let $w:=\gam(x,y,z)$.
By \refLs{L2} and \ref{LgD},
\begin{align}
  [y,w]\cap [w,z]
=  [w,y]\cap [w,z]
\subseteq [x,y]\cap [x,z]
=[x,w]
\end{align}
and thus, since $w\in[x,y]$,
\begin{align}\label{cola}
  [y,w]\cap [w,z]
\subseteq [x,w]\cap [w,y]
=\set{w}.
\end{align}

Define 
$\psi:[0,\dxx yw+\dxx wz]\to T$
as the concatenation $\psi:=\gf_{y,w}*\gf_{w,z}$,
see \eqref{concat}.
Then $\psi$ is continuous, and it follows from \eqref{cola} that $\psi$ is
injective. Consequently, \ref{T2inj} implies,
after a change of variables, 
\begin{align}
  w=\psi(\dxx yw)\in [\psi(0),\psi(\dxx yw+\dxx wz)]=[y,z].
\end{align}

\ref{T2inj}$\iff$\ref{T2dist}: An immediate consequence of \refL{L1}.

\ref{T2inj}$\implies$\ref{T2=}:
Suppose that $\psi:\oi\to T$ is an injective continuous map.
Let $x:=\psi(0)$ and $y:=\psi(1)$.
By \ref{T2inj}, $\psi:\oi\to[x,y]$. Since $\gfxy$ is a homeomorphism, 
$\hzeta:=\gfxy\qw\circ\psi:\oi\to\oxy$ is an injective continuous map with 
$\hzeta(0)=0$ and $\hzeta(1)=\dxy$.
In particular, the image $\hzeta(\oi)$ is connected, and it follows that
$\hzeta(\oi)=\oxy$. Hence $\hzeta$ is a continuous bijection, and thus a
homeomorphism $\oi\to\oxy$; furthermore, $\hzeta$ has to be stricly increasing.
Finally, the definition of $\hzeta$ yields $\psi=\gfxy\circ\hzeta$.

\ref{T2=}$\implies$\ref{T2inj=}: Trivial.

\ref{T2inj=}$\implies$\ref{T2inj}: Trivial.

This completes the proof of \refT{T=}. \qed

\section{Subtrees}\label{Ssub}

We have, as a simple consequence of the definition and \refT{T=} 
a simple result for subsets  of a real tree.

\begin{theorem} \label{Tsubtree}
Let $T$ be a real tree, and let $S\subseteq T$ be a nonempty subset of $T$,
regarded as a metric space with the induced metric. 
Then the following are equivalent.
\begin{romenumerate}
  
\item \label{Tsubt}
$S$ is a real tree.
\item \label{Tsubc}
$S$ is connected.
\item \label{Tsubpc}
$S$ is pathwise connected.
\item \label{Tsubxy}
If\/ $x,y\in S$, then $[x,y]\subseteq S$. (Here $[x,y]$ is taken in the real
tree $T$.)
\end{romenumerate}
\end{theorem}

\begin{proof}
\ref{Tsubt}$\implies$\ref{Tsubpc},
\ref{Tsubxy}$\implies$\ref{Tsubpc}, and
\ref{Tsubpc}$\implies$\ref{Tsubc} are trivial.
  
\ref{Tsubpc}$\implies$\ref{Tsubt}:
Suppose that $S$ is  pathwise connected. Then,
if $x,y\in S$, then there exists a continuous map $\psi:\oi\to S$ with
$\psi(0)=x$ and $\psi(1)=y$. By \ref{T2surj} (for the real tree $T$),
$\psi(\oi)\supseteq[x,y]:=\gfxy\bigpar{[0,\dxy]}$, where $\gfxy$ is the
mapping in \ref{T1} for the real tree $T$.
Hence, $\gfxy:\oxy\to S$, and thus \ref{T1} holds for $S$ too;
uniqueness follows because
$\gfxy$ obviously is unique in $S$ if it is unique in $T$.
Finally, \ref{T2dist} holds in $S$ since it holds in $T$. (In fact, we
could here argue with any of \ref{T2first}--\ref{T2last}.)
Hence, $S$ is a real tree.

\ref{Tsubc}$\implies$\ref{Tsubxy}:
Suppose that $\Tsub$ is connected.
 Let $x,y\in \Tsub$, and consider $[x,y]$ (in the real tree $T$).
Let $z\in(x,y)$, and suppose that $z\notin \Tsub$.
By \ref{T2disco} and \refL{L0}, the components of
$T\setminus{z}$ are disjoint open sets, with $x$ and $y$ in different
components.
Let $U$ be the component containing $x$, and $V$ the union of all other
components; then
$T\setminus{z}= U\cup V$ where $U$ and $V$ 
are  open disjoint subsets 
with
$x\in U$ and $y\in V$. Consequently,
$\Tsub = (\Tsub\cap U)\cup (\Tsub\cap V)$, where $\Tsub\cap U$ and
$\Tsub\cap V$ are two nonempty
disjoint open subsets of $\Tsub$. This contradicts the assumption that
$\Tsub$ is 
connected, and this contradiction shows that $(x,y)\subseteq \Tsub$.
Hence, $[x,y]\subseteq \Tsub$.
\end{proof}

A \emph{subtree} of a real tree is thus a connected nonempty subset.

\begin{theorem}
  The intersection of any family $\set{T_\ga}$ of subtrees of a real tree $T$
  is a subtree of $T$, provided is it nonempty.
\end{theorem}
\begin{proof}
  This is an immediate consequence of \refT{Tsubtree}.
Let $S:=\bigcap_\ga T_\ga$. If $x,y\in S$, then \refT{Tsubtree}\ref{Tsubxy}
shows that $[x,y]\subseteq T_\ga$ for every $\ga$, and thus $[x,y]\subseteq S$.
Hence, another application of \refT{Tsubtree} shows that $S$ is a subtree.
\end{proof}

In particular, it follows that if $T$ is a real tree, then for any nonempty
set $U\subseteq T$, there exists a smallest subtree $S\subseteq T$ with
$U\subseteq S$; we say that $S$ is the subtree \emph{spanned by $U$}.
This subtree can be described as follows.

\begin{theorem}\label{Tspanned}
  Let $T$ be a real tree and let $S$ be the subtree generated by a nonempty
  set $U\subseteq T$. Then
  \begin{align}\label{tspanned1}
    S = \bigcup_{x,y\in U} [x,y].
  \end{align}
Furthermore, for every $x\in U$, we also have
  \begin{align}\label{tspanned2}
    S = \bigcup_{y\in U} [x,y].
  \end{align}
\end{theorem}
\begin{proof}
  Denote the unions in \eqref{tspanned1} and \eqref{tspanned2} by $S'$ and
  $S''_x$, respectively.
Then \refT{Tsubtree}\ref{Tsubxy} shows that $S''_x\subseteq S'\subseteq S$.

On the other hand, $S''_x$ is pathwise connected, since every interval $[x,y]$
is and they contain a common point $x$. Thus \refT{Tsubtree} shows that
$S''_x$ is a subtree. Since $S''_x\supseteq U$, it follows that 
$S''_x\supseteq S$,
and the result follows.
\end{proof}

\section{The four-point inequality}

A different type of characterization of real trees is given by the following
theorem, 
see \eg{} \cite[Theorem 3.40]{Evans} or \cite{Dress} and the references there.
(This characterization is less intuitive, but technically very useful.)
The condition \eqref{4point} is called the 
\emph{four-point inequality} or
\emph{four-point condition};
an equivalent condition is \emph{$0$-hyperbolicity},
see \refD{DGromov} and \refL{L4=}.

\begin{theorem}\label{T4}
A metric space $T$ is a real tree if and only if $T$ is connected and
for any four points $x,y,z,w\in X$
\begin{align}\label{4point}
  d(x,y)+d(z,w) \le \bigpar{d(x,z)+d(y,w)}\bmax\bigpar{d(x,w)+d(y,z)}.
\end{align}
\end{theorem}

\begin{remark}
  It is easily verified that \eqref{4point} is trivial if two or more of
  $x,y,z,w$ coincide; hence it does not matter whether we require $x,y,z,w$
  to be distinct or not.
\end{remark}

\begin{remark}
  By considering all permutations of $x,y,z,w$, it follows that
\eqref{4point} is equivalent to the condition that
(for any $x,y,z,w$), among the three sums
\begin{align}\label{4point2}
    d(x,y)+d(z,w),&& d(x,z)+d(y,w),&&d(x,w)+d(y,z),&
\end{align}
two are equal and the third is equal or less than the other two.
\end{remark}

\refT{T4} is a simple corollary of the following more general result
together with \refT{Tsubtree}. 
For proofs see \eg{} 
\citet[Theorem 8]{Dress},
\citet[\S6.1]{Gromov}, 
or the references mentioned 
in \cite{Dress};
see also \cite{DressMoultonT}.

\begin{theorem}\label{T4+}
Let $X$ be a metric space. Then $X$ can be isometrically embedded into a
real tree if and only if 
the \FPC{}
\eqref{4point} holds
for any four points $x,y,z,w\in X$.
\qed
\end{theorem}

\begin{proof}[Proof of \refT{T4} from \refT{T4+}]
  Suppose that $T$ is connected and that \eqref{4point} holds.
Then, by \refT{T4+}, $T\subseteq\hT$ for some real tree $\hT$.
Since $T$ is connected, $T$ is a real tree by \refT{Tsubtree}.
\end{proof}

Among the consequences we mention the following.

\begin{theorem}
  \label{Tcomplete}
If $T$ is a real tree, then so is its completion $\tT$.
\end{theorem}
\begin{proof}
  By continuity, if the \FPC{} \eqref{4point} holds in the dense subset
  $T$ of $\tT$, then it holds in $\tT$.
Furthermore, since $T$ is connected, so is $\tT$.
hence, $\tT$ is a real tree by \refT{T4}.
\end{proof}

\begin{remark}\label{Rcomplete}
  By \refT{Tcomplete}, 
we may in many situations assume without loss of generality that 
real trees are complete,
since we can replace an arbitrary real tree  by its completion.
\end{remark}

The \FPC{} \eqref{4point} 
can be rewritten in several ways.
Define, for three points $x,y,z$ in a general metric space $(X,d)$,
the \emph{Gromov product}
\begin{align}\label{ak4}
  (x,y)_z:=\tfrac12\bigpar{d(z,x)+d(z,y)-d(x,y)}.
\end{align}
Note that $(x,y)_z\ge0$ by the triangle inequality, and that
\eqref{ak4} meaures how far the triangle inequality is from being an
equality.
Note also that in a real tree, \refL{Lddd} shows that $(x,y)_z=\ddd(z,x,y)$,
which equals the distance from $z$ to $[x,y]$.

\begin{lemma}\label{L4=}
The \FPC{} \eqref{4point} is equivalent to
\begin{align}\label{4pointa}
  (x,y)_w \ge 
(x,z)_w\bmin(y,z)_w
.\end{align}
\end{lemma}
\begin{proof}
By the definition \eqref{ak4}, the inequality \eqref{4pointa}  holds if and
only if at least one of the following holds:
\begin{align}
  d_{x,w}+d_{y,w}-\dxy &\ge d_{x,w}+d_{z,w}-\dxz,\label{l4=a}
\\
  d_{x,w}+d_{y,w}-\dxy &\ge d_{y,w}+d_{z,w}-d_{y,z}.\label{l4=b}
\end{align}
These are equivalent to, respectively,
\begin{align}
d_{y,w}+\dxz &\ge d_{z,w}+\dxy,\label{l4=c}
\\
  d_{x,w}+d_{y,z} &\ge d_{z,w}+\dxy,\label{l4=d}
\end{align}
and thus at least one of them holds if and only iff \eqref{4point} holds.
\end{proof}

We note also that, in fact, it suffices to verify the four-point inequality
for a fixed choice of one of the four points.
\begin{lemma}\label{L4p1}
Let $T$ be a metric space and let $o\in T$ be fixed.
If the four-point inequality \eqref{4point} holds for $w=o$ and
all $x,y,z\in T$, then it holds in general, \ie, for all $x,y,z,w\in T$.
\end{lemma}
\begin{proof}
  By \refL{L4=}, this is the special case $\gd=0$ of \refL{LG2}.
\end{proof}

\section{Rooted real trees}\label{Sroot}

\begin{definition}
  A \emph{rooted real tree} $(T,\rho)$ is a real tree $T$ with a distinguished
  point $\rho\in T$, called the \emph{root}.
\end{definition}

In a rooted real tree $(T,\rho)$, we may define a partial order by
\begin{align}\label{<}
  y \le x \iff y \in[\rho,x],
\qquad x,y\in T.
\end{align}

\begin{theorem}\label{Tpartial}
Let $(T,\rho)$ be a rooted real tree. Then \eqref{<} defines a partial order
in $T$, with $\rho$ as the minimum element.
Moreover, any two points $x,y\in T$ have a greatest common lower bound,
which we denote by $x\land y$. 
Recalling the notation of \eqref{gam} and \refL{Lddd},
we have 
\begin{equation}
x\land y = \gam(\rho,x,y).
\end{equation}
\end{theorem}

\begin{proof}
  It is easily seen, using \refLs{L2} and \ref{L22}, that \eqref{<}
defines a partial order.
It is obvious from \eqref{<} that $\rho\le x$ for every $x\in T$.

For any $x,y\in T$,
by the definition \eqref{<} and \eqref{lgdb},
\begin{align}
  \set{z:z\le x \text{ and } z\le y}
&=
  \set{z:z\le x}\cap
  \set{z:z\le y}
=[\rho,x]\cap[\rho,y]
\notag\\
&=[\rho,\gam(\rho,x,y)]
=   \set{z:z\le \gam(\rho,x,y)},
\end{align}
which shows that $\gam(\rho,x,y)$ is a greatest lower bound $x\land y$.
\end{proof}

For any $x,y\in T$, the path $[x,y]$ from $x$ to $y$ is a combination of the
paths $[x,x\land y]$ and $[x\land y,y]$ (where one or both parts 
may reduce to a single point). Hence, we have
\begin{align}
  [x,y]&=[x,x\land y]\cup[x\land y,y].\label{ord1}\\
d(x,y)&=d(x,x\land y)+d(x\land y,y).\label{ord2}
\end{align}

We note also that for any  subset $\set{x_\ga}_{\ga\in \cA}\subseteq T$,
it follows from  \refT{Tspanned} that
the subtree spanned by these points and the root $\rho$ is
$\bigcup_{\ga}[\rho,x_\ga]$;
see further \refEs{Esub1} and \ref{Esub2}.

\section{Leaves and branch points}\label{Sbranch}

Recall from  \refL{L0} that
the components of $\Tz$ are also the path components of $\Tz$, and that
these are open and  are called 
the \emph{branches} of $T$ at $z$.

\begin{definition}\label{Ddegree}
  Let $T$ be a real tree. 
The \emph{degree} $\deg z=\degg Tz$ of a point $z\in T$ is the number of 
branches at $z$, \ie, the number of components
of $\Tz$. Thus $1\le\deg z\le\infty$ unless $T$ consists of a single point.
\end{definition}

\begin{definition}
Let $T$ be a real tree.
We say that a point $z\in T$ is a \emph{leaf} if $\deg z=1$, 
and a \emph{branch point} if $\deg z\ge3$.

We denote the set of leaves by $T\leaves:=\set{z:\deg z\le1}$.
The \emph{skeleton} of $T$ is the set $T\skel:=T\setminus T\leaves
=\set{z\in T:\deg z\ge2}$,
i.e.,  the set of all
non-leaves of $T$.
\end{definition}

We ignore here the trivial case when $T$ consists of a single point.
(In this case, the point is defined to be a leaf, by a modification of the
definition above, and $T\skel=\emptyset$.)

We note that
\begin{align}\label{skel1}
  T\skel = \bigcup_{x,y\in T}(x,y).
\end{align}

\begin{remark}\label{Rrootleaf}
In a rooted real tree, the root is often not regarded as a leaf, even if its
degree is 1.
\end{remark}

We note also that the branches at a point  can be characterized as follows.

\begin{lemma}\label{Lbranch}
  Let $T$ be a real tree, and let $z\in T$.
Then the following are equivalent, for any $x,y\in\Tz$:
\begin{romenumerate}
    
\item\label{Lbranch1}  
$x$ and $y$ belong to different branches of $T$ at $z$.
\item \label{Lbranch3}
$z\in[x,y]$.
\item \label{Lbranch2}
$[z,x]\cap[z,y]=\set z$.
\end{romenumerate}
\end{lemma}

\begin{proof}
  \ref{Lbranch1}$\iff$\ref{Lbranch3}:
$x$ and $y$ belong to the same path component of $\Tz$ if and only if
$[x,y]\subseteq \Tz$, \ie, if and only if $z\notin[x,y]$.
(Cf.~the condition \ref{T2disco}.)

\ref{Lbranch3}$\iff$\ref{Lbranch2}: 
Immediate from \refLs{LgD} and \ref{L5}.
\end{proof}

\begin{example}\label{Esub1}
  Let $(T,\rho)$ be a rooted real tree and  
let $\set{x_1,\dots,x_n}$ be a finite set of points in $T$.
By \refT{Tspanned},
the subtree $T_1$  spanned by $\set{x_1,\dots,x_n}$ and the root $\rho$ is
$\bigcup_{i=1}^n[\rho,x_i]$. 
It is easily seen that the leaves of $T_1$ are $\rho$ (but see \refR{Rrootleaf})
and the set of maximal elements of $\set{x_1,\dots,x_n}$, i.e.,
$\set{x_i:x_i\notin[\rho,x_j]\text{ for every }j\neq i}$.
Furthermore the branch points of $T_1$ are a subset of $\set{x_i\land x_j:
  i\neq j}$. The sets of leaves and branch points are thus finite.
\end{example}

\section{A metric space of compact real trees}\label{SGH}
Consider the set $\bbT$ of all compact real trees, or rather the set of all
equivalence classes under isometry of compact real trees (so that two isometric
real trees are regarded as the same).
(The set theoretic difficulties with ``all compact real trees'' are handled
in the standard way: since a compact real tree, as any compact metric space,
has cardinality at most $\mathsf c$, it suffices to consider real trees that
as sets are subsets of, for example, $\bbR$.)

The set $\bbT$ can be equipped with a metric, the 
\emph{Gromov--Hausdorff  distance}, 
which makes $\bbT$  a complete separable metric space. 
Similarly, the set $\bbT_1$ of rooted compact real trees is a complete
separable  metric space,
equipped with (a rooted version of) the Gromov--Hausdorff  distance.
See \cite{EvansPW}  for definitions and proofs;
see also  \cite[Section 7.3]{Burago} for the Gromov--Hausdorff distance for
general metric spaces.

The fact that $\bbT$ and $\bbT_1$ thus are complete separable metric spaces
(and thus Polish topological spaces) makes it possible to define random
compact real trees as random elements of one of these spaces, and a lot of
standard machinery then is available.

For noncompact real trees, one can similarly 
use the version of Gromov--Hausdorff convergence in \cite[Section 8.1]{Burago}.

\section{Some examples}
\begin{example}\label{E1}
A \emph{combinatorial tree} is a non-empty set $V$ of vertices 
(finite or infinite)
together with
a set $E$ of unorded pairs $\set{v,w}$ of vertices, such that $(V,E)$ is a
tree in the usual combinatorial sense. 
(An edge \set{vw} is often denoted $vw$ for simplicity.)

We may regard a combinatorial tree as a real tree $T$, by regarding each
edge as a copy of 
$\oi$, with the endpoints identified with the corresponding vertices in $V$.
Equivalently, we may define $T$ as the disjoint union of $V$ and one copy of
$(0,1)$ for each edge in $E$, with a suitably defined metric.
(We omit the details, and the verification that $T$ is a real tree.)
In any case, we regard $V$ as a subset of $T$.

Note that for $v,w\in V$, the distance $d(v,w)$ equals the usual distance in
a graph, i.e., the number of edges in a shortest path from $v$ to $w$.

The degree $\degg Tz$ of a vertex $z\in V$ equals the degree of $z$ in the
graph $(V,E)$; the degree of any vertex in $T\setminus V$ is 2.
In particular, 
the leaves of $T$ are precisely the leaves of the tree $(V,E)$
(i.e., the vertices in $V$ adjacent to a single edge in $E$),
and the branch points are the vertices in $V$ that have degree $\ge3$.

It is easy to see that $T$ always is complete,  that 
$T$ is separable if
and only if $V$ (and thus also $E$) is countable, and that
$T$ is compact if
and only if $V$ (and thus also $E$) is finite.  
\end{example}

\begin{example}\label{E2}
  More generally, suppose as in \refE{E1} that $(V,E)$ is a combinatorial
  tree,
and assume also that for every edge $e\in E$, we are given a real number
$\ell_e$, called the \emph{length} $e$.
We may construct a real tree $T$ as in \refE{E1}, but now for each edge
  $e$ taking an interval of length $\ell_e$.
(In particular, $\ell_e=1$ for all $e$ gives back the real tree in
\refE{E1}.)

We see again that $T$ is separable if and only if $V$ is countable.
(In one direction, note that if $D$ is a countable dense subset of $T$, then
every edge contains, in its interior, an element of $D$; hence $E$ is
countable.) 

Moreover, $T$ is compact if $V$ is finite,
but the converse does not
hold. One counterexample is an infinite star which is compact for some
(but not all) choices of edge lengths: let $V=\set{0,1,\dots}$
and  $E=\set{0i:i\ge1}$, with length $\ell_{0i}=2^{-i}$.
\end{example}

\begin{example}\label{Esub2}
As in \refE{Esub1}, let $T_1$ be the subtree of a rooted real tree that is
spanned by a finite set of points \set{x_1,\dots,x_n} and the root. 
It follows from \refE{Esub1} and \eqref{ord1}--\eqref{ord2} that
the real tree $T_1$ can be constructed as in \refE{E2} from 
a finite combinatorial tree $(V,E)$ where 
$V=\set{x_i}\cup\set{x_i\land x_j:1\le i< j\le n}\cup\set\rho$, and a
suitable set of edges $E$ with suitable lengths $\ell_e$; we omit the details.  
\end{example}

\begin{example}\label{E3}
  The \emph{infinite binary tree} is a combinatorial tree
with $V:=\bigcup_{n=0}^\infty\setoi^n$, the set of all finite strings
from $\setoi$ (including the empty string $\emptyset$); 
the edges are all pairs of the type
\set{v,v0} and \set{v,v1} for strings $v\in V$.
Let $(\ell_n)_1^\infty$ be a sequence of positive real numbers, and
let $T$ be the real tree constructed in \refE{E2} with edge lengths
defined by $\ell_{\set{v,vj}}:=\ell_{|v|+1}$, where $|v|$ is the length of the
string $v$.

No vertex in $V$ is a leaf, and thus, see \refE{E1},
the real tree $T$ has no leaf, so $T\leaves=\emptyset$ and $T\skel=T$.

The real tree $T$ is always separable, and never compact.
It is easy to see that $T$ is complete if $\sum_n\ell_n=\infty$, but not if
$\sum_n\ell_n<\infty$, since in the latter case, the sequence
$(0^n)_0^\infty=\emptyset,0,00,000,\dots$ is a Cauchy sequence without a limit.
See further the next example.
\end{example}

\begin{example}
  \label{E4}
Let $T_0$ be the infinite binary tree in \refE{E3} and assume that
$L:=\sum_n\ell_n<\infty$. 
Note that $d(\roott,z)<L$ for every $z\in T_0$.
For every $s<L$, the set \set{z\in T_0:d(\roott,z)\le s} is closed and 
contained in a finite number of edges, and thus it is compact.

Let $(z_n)\xoo$ be a Cauchy sequence in $T_0$. Then the sequence
$d(\roott,z_n)$ is a Cauchy sequence, so it converges to some limit
$d_\infty\le L$.
If the limit $d_\infty<L$, then we see that the Cauchy sequence $(z_n)$
belongs to a compact subset of $T_0$, and thus it converges.
On the other hand, if $d(\roott,z_n)\to L$, then the Cauchy sequence cannot
converge, since a limit $z$ would have to satisfy $d(\roott,z)=\lim_\ntoo
d(\roott,z_n) = L$, but no such $z$ exists in $T_0$.

Consider now the completion $T$ of $T_0$; $T$ is a real tree by
\refT{Tcomplete}, and we call $T$ a
\emph{complete infinite binary tree}.
We claim that $T\setminus T_0$ may be identified with the
set $\setoi^\infty$ of infinite strings from \setoi. In fact, if
$v=\xi_1\xi_2\cdots\in \setoi^\infty$, 
then let  $v_n:=\xi_1\cdots\xi_n\in V$ for each $n\ge0$;
we have $d(v_n,v_m)=\sum_{n<i\le m} \ell_i$ when $n\le m$, and thus $(v_n)$ is
a Cauchy sequence in $T_0$ so it has a limit in $T$ which we represent by
$v$.
We have $d(\roott,v)=L$, and thus $v\in T\setminus T_0$.

Furthermore,
if $v=\xi_1\xi_2\cdots$ and $v'=\xi'_1\xi'_2\cdots$ are elements of
$\setoi^\infty$, let $D(v,v'):=\inf\set{i:\xi_i\neq\xi'_i}$; thus
$1\le D(v,v')<\infty$ if $v\neq v'$, but  $D(v,v)=\infty$.
Further, let $\LL_n:=\sum_{i\ge n}\ell_i$; thus $\LL_1=L$, and
$\LL_n\downto0$ as \ntoo. It is then easy to see that for any
$v,v'\in\setoioo$, regarded as elements of $T$, we have
\begin{align}\label{dLL}
d(v,v')=2\LL_{D(v,v')}.  
\end{align}
In particular, this shows that two different strings in $\setoioo$ represent
different points in $T$, so we may regard $\setoioo$ as a subset of $T$.
Note also that, since $\LL_n\to0$ as \ntoo, the metric \eqref{dLL} induces
the product topology on $\setoioo$; thus $\setoioo$ is a compact subset of
$T$, homeomorphic to the Cantor set.

Finally, if $(z_n)$ is any Cauchy sequence in $T_0$ without limit in $T_0$, we
have seen that $d(\roott,z_n)\to L$, and since $\ell_k\to0$ as \ktoo,
it follows easily that we may approximate each $z_n$ by $z_n'\in V$ such
that
$d(z_n,z_n')\to0$ as \ntoo.
Then $z_n'$ is a finite string; we extend it (arbitrarily) to an infinite
string $z_n''\in\setoioo$ and note that
\begin{align}
  d(z_n',z_n'')=\LL_{|z_n'|+1}\to0.
\end{align}
Hence $d(z_n,z_n'')\to0$, and thus also $(z_n'')$ is a Cauchy sequence;
furthermore $(z_n'')$ lies in the
compact metric space $\setoioo$. Consequently $z_n''\to z$ for some
$z\in\setoioo$, and thus also $z_n\to z$. This shows that every Cauchy
sequence in $T_0$ has a limit either in $T_0$ or in $\setoioo$, and thus
$T=T_0\cup\setoioo$ as claimed above.

The complete infinite binary tree $T$ is compact; this follows either by
using a modification of the argument above to show that an arbitrary
sequence $(z_n)$ in $T$ has a subsequence that converges, or by noting that
for every $\eps>0$, there is a finite $\eps$-net in $T$, since $\setoioo$ is
compact, and so is the set $\set{z\in T:d(z,\setoioo)\ge\eps}$; we omit the
details. 

Note that  the complete infinite binary tree $T=T_0\cup\setoioo$ 
regarded as a set is the same for
every sequence $(\ell_n)\xoo$ satisfying the assumptions $\ell_n>0$ and
$\sum_n\ell_n<\infty$; furthermore, it is easily seen that
the topology of $T$ is the same for all such $(\ell_n)\xoo$.
However, the metric on $T$
depends on $(\ell_n)\xoo$, as
is seen \eg{} by \eqref{dLL}.

It is easily seen that the set of leaves $T\leaves=\setoioo$, and thus the
skeleton $T\skel=T_0$.
\end{example}

\begin{example}\label{Eg}
  Let $\ell>0$ and let
$g:\oell\to[0,\infty)$ be a non-negative continuous function defined 
on $\oell$ with $g(0)=g(1)=0$.
We define a semimetric $d=d_g$ on $\oell$ by
\begin{align}
  d(s,t):=g(s)+g(t)-2\min_{u\in[s,t]} g(u),
\end{align}
when $s,t\in\oell$ with $s\le t$ (and, of course, $d(s,t):=d(t,s)$ when $s>t$).
It is easily verified that this is a semimetric; thus, if we define an
equivalence relation on $\oell$ by $s\equiv t$ if $d(s,t)=0$, then the
quotient space $T_g:=\oell/\equiv$ is a metric space; moreover, 
it is not difficult to show that $T_g$ is connected and satisfies the
4-point inequality; thus $T_g$ is a real tree by \refT{T4}.
The quotient map $\oell\to T_g$ is continuous, and thus
$T_g$ is a compact real tree.

Note that if $s\le t$, then $s\equiv t \iff g(u)\ge g(s)=g(t)$ for every
$u\in[s,t]$. 
(Informally, we  may think of obtaining $T_g$ by
putting glue on the downside of the graph of $g$, and then compressing the
$x$-axis.)

As a simple example, a finite combinatorial tree as in \refE{E1} or \ref{E2}
can be constructed in this way by taking $g(t)$ to be the \emph{contour}
function of the tree, defined as the height (distance to the root) of a
particle that moves with unit speed along the ``outside'' of the tree,
starting and ending at the root.

In fact, every compact rooted real tree may be constructed in this way (up
to isometry) using a suitable function $g:\oi\to\ooo$ 
\cite[Remark 3.2]{LeGallMiermont}.

In applications, as in the following two examples, 
$g(t)$ is usually a random function, and then $T_g$ is a
random real tree.
If we for simplicity let $\ell$ be fixed (for example, $\ell=1$), then the map
$g\mapsto T_g$ is a continuous map from $C[0,\ell]$ to the set $\bbT_1$
of rooted
compact real trees with the Gromov--Hausdorff metric in \refS{SGH}, 
see \cite{DuqLeGall2005}.
In particular, this map is (Borel) measurable, so if $g$ is a random element
of $C[0,\ell]$, then $T_g$ is a well-defined
random element of the Polish space $\bbT_1$ of rooted compact real trees.
\end{example}

\begin{example}\label{EBrownian}
The \emph{Brownian continuum random tree},
originally constructed (in several different ways) by
\citet{AldousI,AldousII,AldousIII}, 
is the random real tree $T_\bex$ obtained by the construction in \refE{Eg}
letting $g(t)$
be a random (normalized) Brownian excursion $\bex:\oi\to\ooo$;
see \cite[Corollary 22]{AldousIII}.
(Actually, Aldous defined the Brownian continuum random tree to be 
$T_{2\bex}$ in our notation, 
but the convention has later changed to $T_\bex$; of course, the
results differ only by a scaling.)
See  \eg{} 
\cite{AldousI,AldousII,AldousIII},
\cite{Evans} and \cite{LeGall2006} 
for properties of this random real tree.
In particular, $T_\bex$ has almost surely a countably infinite number of
branch points, all of degree 3, and an uncountable number of leaves.
\end{example}

\begin{example}\label{ELevy}
 More generally, a \emph{\Levy{} tree} is a random real tree constructed as
 in \refE{Eg} letting $g$ be a random continuous fuction known as 
the \emph{height process} of a \Levy{} process (with certain conditions),
see \cite{DuqLeGall2002,DuqLeGall2005}.
In the special case when the \Levy{} process is Brownian motion, this height
process is a Brownian excursion and we obtain the Brownian continuum random
tree as in \refE{EBrownian}.

Other special cases are the
\emph{stable trees}, see \cite{LeGall2006}.
\end{example}

\begin{example}\label{Epartial}
  Let $T$ be a partially ordered set such that
  \begin{romenumerate}
  \item 
Any two elements $x,y\in T$ have a greatest lower bound $x\land y$.
  \item 
For every $x\in T$, 
the set $L_x:=\set{y\in T:y\le x}$ is linearly ordered.
  \item 
There is a
\emph{height function} $h:T\to\bbR$ such that for every $x\in T$,
the restriction $h:L_x\to\bbR$ is an order-preserving 
bijection onto an interval $(a,h(x)]$ or
$[(a,h(x)]$ for some $a\in\bbR\cup\set{-\infty}$ (where thus the interval is
$(-\infty,h(x)]$ if $a=-\infty$).
  \end{romenumerate}
It is easily seen from (i) and (ii) that $a$ in (iii) cannot depend on $x$.
Moreover, either $h(L_x)=[a,h(x)]$ for all $x$, and then $T$ has a smallest
element $o$ with $h(o)=a$, or
$h(L_x)=(a,h(x)]$ for every $x$, and then $T$ has no minimum (or minimal)
element. 

Define
\begin{align}\label{dpartial}
  d(x,y):=h(x)+h(y)-2h(x\land y),
\qquad x,y\in T.
\end{align}
It is easily seen that $d$ is a metric on $T$, which makes $T$ a real tree.
The path $[x,y]$ between two points $x,y\in T$ consist of the two parts
$[x,x\land y]$ and $[x\land y,y]$, which are subsets of $L_x$ and $L_y$,
respectively.

If $T$ is has a minimum element $o$, we choose $o$ as a root, and then 
the partial order defined in \eqref{<} is the original order.
Moreover, $h(x)=d(x,o)+a$ with $a:=h(o)$.

Conversely, 
if $(T,\rho)$ is a rooted real tree, the partial order defined in
\eqref{<} satisfies (i)--(iii) above with the height function
$h(x):=d(x,\rho)$, and the construction above  returns the original metric
on $T$.

It is easily verified that the trees constructed in \refE{Eg} are of this
type, with height function $g$
(after identifying equivalent points).
\end{example}

\begin{example}\label{Ebig}
  Let $T$ be the collection of all bounded non-empty subsets of $R$ that
  contain their supremum. Let $h(A):=\sup A$ for $A\in T$, and define a
  partial order by letting $A\le B$ if $A=(B\cap(-\infty,t])\cup\set{t}$ for
  some $t\in\bbR$ with $t\le h(B)$ (necessarily $t=h(A)$.

It is easily verified that $\le$ is a partial order, and that it satisfies
(i)--(iii) in \refE{Epartial} with the height function $h$ (with $a=-\infty$).
Hence, \eqref{dpartial} defines a metric that makes $T$ inte a real tree.

Note that this is a very large tree. Its cardinality is $2^{\mathsf c}$, and every
point in $T$ has uncountable degree (more precisely, also of cardinality
$2^{\mathsf c}$).
In particular, $T$ is not separable.

See \cite[Examples 3.18 and 3.45]{Evans} for further properties of this real
tree. 
\end{example}

\section{The length measure}\label{Slength}

Every real tree has a natural measure on it, defined as follows.
(See \eg{} \cite[2.10]{Federer} for the definition and properties of
Hausdorff measures.)
We  let $\cB(T)$ denote the collection of Borel subsets of $T$.

\begin{definition}
  Let $T=(T,d)$ be a real tree.
The \emph{length measure} $\gl$ on $T$ is the $1$-dimensional
Hausdorff measure $\hmi$ on the skeleton
$T\skel$, regarded as a Borel measure on $T$.
In other words, for a Borel set $A\in\cB(T)$, 
\begin{align}\label{lm}
  \gl(A):=\hmi(A\cap T\skel),
\end{align}
where $\hmi$ is the Hausdorff measure on $T\skel$.
\end{definition}

Note that, by definition, $\gl(T\leaves)=0$. 

In the definition \eqref{lm}, if $A$ is a Borel set in $T$,
then $A\cap T\skel$ is
a Borel subset of $T\skel$, and thus $\gl(A)$ is well defined.
If $T\skel$ is a Borel subset of $T$, or more generally 
a $\hmi$-measurable subset of $T$, 
then we can also define the length measure by
\eqref{lm} interpreting $\hmi$ as the Hausdorff measure on $T$.
In general, $T\skel$ is \emph{not} measurable
(see \refE{Enonmeas} below),
but it is in
most cases of interest (and in particular for all compact $T$) by 
\refT{Tskel} below.
Alternatively, we can always (even if $T\skel$ is not measurable)
define $\gl$ by \eqref{lm} interpreting $\hmi$ as the \emph{outer} Hausdorff
meaure on $T$.

\begin{remark}
We have here defined $\gl$ as a Borel measure.
Alternatively, we may more generally define it by \eqref{lm} for every
$A\subseteq T$ such that $A\cap T\skel$ is a $\hmi$-measurable subset of
$T\skel$. 
\end{remark}

We note some elementary properties of $\gl$, which justify the name length
measure. 
Note that for every $x,y\in T$, the set $[x,y]$ is isometric to the 
interval $[0,d(x,y)]\subset\bbR$, and is thus compact.
Hence, $[x,y]$ and $(x,y)=[x,y]\setminus\set{x,y}$ are Borel sets in $T$.

\begin{theorem}\label{TL1}
The length measure $\gl$ is a continuous measure on $T$, i.e., 
  if $x\in T$, then $\gl\set x=0$.
Moreover, if $x,y\in T$, then
\begin{align}\label{tl1}
  \gl([x,y])=\gl((x,y))=d(x,y).
\end{align}
\end{theorem}

\begin{proof}
  For every $x\in T$, we have $\gl\set x=0$ by \eqref{lm}.

Hence, for every $x,y\in T$, $\gl([x,y])=\gl((x,y))$.
Furthermore, $(x,y)$ is a subset of $T\skel$ isometric to 
the interval $(0,d(x,y))\subset\bbR$, and thus
$\hmi((x,y))=\hmi((0,d(x,y))=d(x,y)$, 
since the Hausdorff measure $\hmi$ on $\bbR$ equals the 
Lebesgue measure.
\end{proof}

When $T$ is separable, we can say more.

\begin{theorem}\label{Tskel}
Suppose that $T$ is a separable real tree.
\begin{romenumerate}
\item\label{Tskel1} 
Then $T\leaves$ and $T\skel$ are Borel subsets of $T$.

\item\label{Tskel3} 
The length measure $\gl$ is $\gs$-finite.  

\item \label{Tskel2}
The length measure is the unique Borel measure $\gl$ on $T$ with
\begin{align}\label{ts}
  \gl([x,y])=d(x,y),
\qquad x,y\in T,
\end{align}
and $\gl(T\leaves)=0$.
\end{romenumerate}
\end{theorem}

\begin{proof}
\pfitemref{Tskel1}
  Let $D$ be a countable dense subset of $T$.
It is easy to see that then, \cf{} \eqref{skel1},
\begin{align}\label{skel11}
  T\skel = \bigcup_{x,y\in D}(x,y).
\end{align}
Furthermore, as noted above, $(x,y)=[x,y]\setminus\set{x,y}$ is a Borel set
for any $x$ and $y$. Hence, $T\skel$ is Borel,
and thus so is its complement $T\leaves$.

\pfitemref{Tskel3}
This follows from \eqref{skel11}, since 
$\gl$ is concentrated on $T\skel$ by \eqref{lm},
and $\gl((x,y))<\infty$ for each $(x,y)$ by \eqref{ts}.

\pfitemref{Tskel2}
We have already shown \eqref{ts} in \refT{TL1}, and $\gl(T\leaves)=0$
follows directly from \eqref{lm}.

To show uniqueness, suppose that $\gl'$ is another Borel measure on $T$ with
$\gl'([x,y])=d(x,y)$ for all $x,y\in T$, and $\gl'(T\leaves)=0$.
Note first that then $\gl'\set{x}=0$ for every $x\in T$, and thus for all
$x,y\in T$, we have, by the assumption and \eqref{ts},
\begin{align}\label{ts4}
  \gl'((x,y))=\gl'((x,y])=\gl'([x,y])
=
  \gl((x,y))=\gl((x,y])=\gl([x,y])
<\infty
.\end{align}

We may assume that $T\skel\neq\emptyset$.
Let $x_0,x_1,\dots$ be a dense subset of $T\skel$, and let, for $n\ge1$,
\begin{align}\label{ts2}
  T_n:=\bigcup_{i=1}^n[x_0,x_i]\subset T\skel
.\end{align}
$T_n$ is a pathwise connected subset of $T$, and thus $T_n$ is a real tree
by \refT{Tsubtree}. (In fact, by \refT{Tspanned}, $T_n$ is the subtree spanned
by $x_0,\dots,x_n$.)
We consider the restrictions of $\gl$ and $\gl'$ to the compact (and thus
Borel) subset $T_n$.

First, for every $n\ge2$ we have $[x_0,x_n]\cap T_{n-1}=[x_0,y_n]$ for some 
$y_n\in T_{n-1}$, and then $T_n=T_{n-1}\cup (y_n,x_n]$ is a partition into
two disjoint Borel subsets; hence, induction and \eqref{ts4} yield
\begin{align}\label{ts6}
  \gl'(T_n)=\gl(T_n)<\infty,
\qquad n\ge1.
\end{align}
It now follows by the monotone class theorem 
(see \eg{} \cite[Theorem 1.2.3]{Gut})
that 
\begin{align}\label{ts7}
  \gl'(A)=\gl(A),
\qquad A\in\cB(T_n),
\end{align}
because 
it follows from \eqref{ts6} that the collection
$\cD_n:=\set{A\in\cB(T_n):\gl'(A)=\gl(A)}$ is a Dynkin system,
the collection $\cA_n:=\set{[x,y]:x,y\in T_n}\cup\emptyset$
is a $\pi$-system (i.e., closed under finite intersections)
that generates $\cB(T_n)$,
and $\cA_n\subseteq\cD_n$ by \eqref{ts4}.

Now let $n\to\infty$; then $T_n\upto \bigcup_1^\infty T_n=T\skel$.
Hence it follows from \eqref{ts7} that for every $A\in\cB(T)$,
\begin{align}\label{ts8}.
  \gl'(A\cap T\skel)=\lim_\ntoo \gl'(A\cap T_n)
=\lim_\ntoo \gl(A\cap T_n) = \gl(A\cap T\skel).
\end{align}
Finally, by assumption and definition $\gl'(A\cap T\leaves)=\gl(A\cap
T\leaves)=0$, 
and thus \eqref{ts8} yields
\begin{align}\label{ts9}.
\gl'(A)=  \gl'(A\cap T\skel)
= \gl(A\cap T\skel)
=\gl(A),
\qquad A\in\cB(T).
\end{align}
\end{proof}

\begin{example}\label{Edimension}
  Let $T$ be the complete infinite binary tree in \refE{E4}, with
  $\ell_n:=2^{-\gam n}$ for some $\gam>0$.
It follows from \eqref{dLL} that for every $k\ge0$, $T\leaves=\setoioo$ can be
partitioned into $2^k$ disjoint balls of radius $c 2^{-\gam k}$,
where $c=2(1-2^{-\gam})\qq$.
It follows by standard arguments (see \eg \cite[Section 1.2]{BishopPeres})
that the $(1/\gam)$-dimensional Hausdorff measure $\hmx{1/\gam}(\setoioo)$
is finite and positive, and thus $\setoioo$ has Hausdorff dimension
$1/\gam$.
(The Minkowski dimension 
\cite[Section 1.1]{BishopPeres}
of $\setoioo$ is the same.)

The Hausdorff dimension of $T\leaves=\setoioo$ 
can thus be any number in $(0,\infty)$.
We see also that $\hmi(T\leaves)=\hmi(\setoioo)>0$ when $\gam\le1$
(and $\infty$ when $\gam<1$);
this shows that in general, the length measure is not equal to 
the Hausdorff measure $\hmi$ on $T$.

We see also that the total length measure $\gl(T)=\sum_1^\infty 2^{n-\gam n}$ is 
finite for $\gam>1$ but infinite for $\gam\le1$.
\end{example}

\begin{example}\label{Enonmeas}
Consider the complete infinite binary tree $T$ in \refE{E4}, for some
sequence $(\ell_n)\xoo$, and
  let $A\subseteq T\leaves$ be an arbitrary subset of $T\leaves=\setoioo$. 
 Define the real tree  $T_A$ by attaching an interval $[x,x']$ of length $1$
to every $x\in A$ in the obvious way.

Then $T_A$ is a real tree with
\begin{align}
T_A\leaves&=\set{x\in\setoioo:x\notin A}\cup\set{x':x\in A}
\notag\\&
=\bigpar{\setoioo\setminus A}\cup\set{x':x\in A}.
\end{align}
Suppose that $T_A\leaves $ is a Borel set in $T_A$. Then so is
\begin{align}
T_A\leaves\cap\set{x\in T_A:d(\roott,x)=L}
=\setoioo\setminus A.
\end{align}
Hence, $A$ then is a Borel set in $\setoioo$.
Consequently, if we choose $A$ as a subset of $\setoioo$ that is not a Borel
set, then $T_A\leaves$ and its complement $T_A\skel$ are not Borel sets in
$T_A$. 

If we specialize to the case $\ell_n=2^{-n}$, then (see \refE{Edimension})
$\setoioo$ has finite and positive Hausdorff measure $h:=\hmi(\setoioo)$.
Since $\setoioo$ is a Polish space, and $\hmi$ is continuous and with full
support,
it follows that the measure space
$(\setoioo,\hmi)$ is isomorphic to $(\oi,h\nu)$, where $\nu$ denotes the
Lebesgue measure; this holds both for the Borel \gsf{s} and their
completions,
which are the $\hmi$-measurable sets in $\setoioo$ and the Lebesgue
measurable sets in $\oi$, respectively.
Hence, there exists $A\subset\setoioo$ such that $A$ is not measurable for
$\hmi$, and it follows that $T_A\leaves$ and $T_A\skel$ are not measurable
for the Hausdorff measure $\hmi$.
\end{example}

\section{Leaf measure}\label{Sleafm}

In some applications, a real tree is equipped with a different (Borel) measure,
which, in contrast to the length measure in \refS{Slength}, is supported on
the set of leaves $T\leaves$. For simplicity, suppose that $T$ is a
separable real tree, so that $T\leaves$ is a Borel set by \refT{Tskel}.
We then may call any Borel measure supported on $T\leaves$ a \emph{leaf
  measure}. 
Note that a leaf measure thus has to be specified, and is not automatically
determined by $T$, unlike the length measure in \refS{Slength}.

Usually, one considers leaf measures that are probability measures; they
thus give a meaning to ``a random leaf''.

\citet{AldousIII} defines a \emph{continuum tree} as a rooted real tree
(with some extra conditions) equipped with a nonatomic
probability measure that is 
supported on the set of leaves (and thus a leaf measure in our sense), and 
furthermore has full support in the sense that for any $x$ in the skeleton
of the tree, the set $\set{y:y>x}$ (recall \eqref{<}) has positive measure.

\begin{example}
  If $T$ is constructed from a finite combinatorial tree as in \refE{E1} or
  \ref{E2}, then $T$ has a finite number of leaves, and a natural leaf
  measure is given by the uniform distribution on $T\leaves$.
\end{example}

\begin{example}
  If $T$ is constructed from a continuous function $g:[0,\ell]\to\ooo$ as in
  \refE{Eg}, then the natural (quotient) mapping $[0,\ell]\to T$ is
  continuous, and thus measurable, so it maps the Lebesgue measure on $[0,\ell]$
to a measure $\mu$ on $T$. This is in general not a leaf measure
(one counterexample is when $g$ is the contour function of a finite
combinatorial tree as in \refE{E1} or \ref{E2}; then the measure $\mu$
is 2 times the length measure on $T$, as is easily seen).
However, $\mu$ is a leaf measure under some conditions
\cite[Theorem 13]{AldousIII};
in particular, $\mu$ is 
almost surely 
a leaf measure
when, as in \refE{EBrownian}, $g=\bex$, the (random) standard
Brownian excursion;
see \cite[Corollary 22]{AldousIII}. 
\end{example}

\appendix
\section{Gromov hyperbolic spaces}
This appendix is a long remark on an approximate version of the \FPC,
which is important in other contexts; 
the appendix can be skipped by those interested in trees only.
$X=(X,d)$ is a metric space.

\citet{Gromov} 
defined \emph{$\gd$-hyperbolic metric spaces} as metric spaces
where the four-point inequality holds up to an error $\gd$.
More precisely, he used the version \eqref{4pointa}, and the definition is
as follows, for a given $\gd\ge0$:
\begin{definition}[Gromov]\label{DGromov}
  A metric space $X=(X,d)$ is  $\gd$-hyperbolic if, for all $x,y,z,w\in X$,
\begin{align}\label{4gda}
  (x,y)_w \ge 
{(x,z)_w\bmin(y,z)_w}
-\gd
.\end{align}
\end{definition}

By \refL{L4=}, \refT{T4+} can be formulated as follows.
\begin{theorem}
  \label{T4G}
A metric space is $0$-hyperbolic if and only if it can be isometrically
embedded in a real tree.
\end{theorem}

However, the main interest of \refD{DGromov} is for $\gd>0$.
The value of $\gd$ is often not important, and one says that a metric space
$X$ is 
\emph{Gromov hyperbolic} if it is $\gd$-hyperbolic for some $\gd<\infty$.

\begin{remark}
Intuitively, at large distances, the $\gd$ in \eqref{4gda} is insignificant,
and thus the large scale geometry of a Gromov hyperbolic space is
``tree-like''.
\end{remark}

There are several alternatives to \refD{DGromov},
with conditions that are equivalent in the sense that if one holds, then so
do the others, 
but with $\gd$ replaced by $C\gd$ for some (small) constant $C$.
(The conditions are in general not equivalent with the same $\gd$.)
Some of these alternative conditions are given (or implicit)
in the following lemmas.
See further \eg{} \cite{Gromov}, \cite{Burago} and \cite{Vaisala}.

\begin{lemma}\label{LG1}
  The inequality \eqref{4gda} holds if and only if
\begin{align}\label{4pointgd}
  d(x,y)+d(z,w) \le \bigpar{d(x,z)+d(y,w)}\bmax\bigpar{d(x,w)+d(y,z)} +2\gd.
\end{align}
\end{lemma}
\begin{proof}
  As the proof of \refL{L4=}, adding $2\gd$ to the \lhs{s} of
  \eqref{l4=a}--\eqref{l4=d}. 
\end{proof}

\begin{lemma}[\cite{Gromov}]\label{LG2}
Let $X$ be a metric space and 
let $o\in X$ be fixed.
Let $\gd\ge0$.
  If \eqref{4gda} holds for $w=o$ and all $x,y,z\in X$, then it holds
for all $x,y,z,w\in X$ with $\gd$ replaced by $2\gd$.
\end{lemma}

\begin{proof}
We first show that the assumption implies that, for any $x,y,z,w$,   
\begin{align}\label{gr1}
  (x,y)_o+(z,w)_o 
\ge \bigpar{(x,z)_o+(y,w)_o}\bmin \bigpar{(x,w)_o+(y,z)_o}
-2\gd
.\end{align}
To see this, we first note that both sides are symmetric in $x$ and $y$, and
also in $z$ and $w$; hence, by interchanging $(x,y)$ and/or $(z,w)$ if
necessary, we may assume that $(x,z)_o$ is the largest of the four numbers
$(x,z)_o,(x,w)_o,(y,z)_o$, $(y,w)_o$. In this case, the assumption implies
\begin{align}\label{gr2}
  (x,y)_o &\ge (x,z)_o\bmin(y,z)_o-\gd
= (y,z)_o-\gd,
\\\label{gr3}
(z,w)_o&\ge  (x,z)_o\bmin(x,w)_o-\gd
= (x,w)_o-\gd,
\end{align}
and thus
\begin{align}\label{gr4}
  (x,y)_o+(z,w)_o
\ge  (y,z)_o+(x,w)_o-2\gd,
\end{align}
verifying \eqref{gr1}.

Next, by exanding all Gromov products in \eqref{gr1} according to the
definition \eqref{ak4}, we see that \eqref{gr1} does not depend on the
choice of $o$; hence \eqref{gr1} holds for every $o,x,y,z,w\in X$.

In particular, we can choose $o=w$ in \eqref{gr1}. Since $(w,v)_w=0$ for
every $v$ by \eqref{ak4}, we obtain
\begin{align}\label{gr5}
  (x,y)_w \ge (x,z)_w\bmin(y,z)_w-2\gd, 
\end{align}
as asserted.
\end{proof}

A \emph{geodesic} in a metric space $X$ is an isometric curve, \ie, an isometric
mapping of an interval $I\subseteq\bbR$ into $X$.
If a geodesic $\gf$
is defined on an interval $I=[a,b]$ that is closed and finite, we
say that the geodesic has the endpoints $\gf(a)$ and $\gf(b)$, and that it
\emph{joins} $\gf(a)$ and $\gf(b)$. 
(Then, the geodesic necessarily has length $\dxy$,
and we may choose $I=\oxy$.)

Note that condition \ref{T1} says that for every $x,y\in X$, there is  a
unique geodesic from $x$ to $y$ (provided we normalize $I=\oxy$).

More generally,
a metric space $X=(X,d)$ is \emph{geodesic}, if every pair $x,y\in X$
is joined by a geodesic.
In other words, we assume the existence part of \ref{T1}, but
uniqueness is not assumed. In particular, every real tree is geodesic.

Gromov hyperbolicity is often studied for geodesic metric spaces.
This can be done without real loss of generality by the following result
by \citet{BS}, to which we refer for a proof.

\begin{theorem}[{\citet[Theorem 4.1]{BS}}]\label{TBS}
Let $\gd\ge0$.
  A metric space is $\gd$-hyperbolic if and only if it can be isometrically
  embedded into a $\gd$-hyperbolic complete geodesic metric space.
\nopf
\end{theorem}

In particular, for $\gd=0$, we recover \refT{T4+} (together with
\refT{Tcomplete}). 

For geodetic metric spaces, there are further conditions equivalent to
Gromov hyperbolicity.
We note first an extension of \refL{L1} to general geodesic spaces.

\begin{lemma}\label{LH3}
  Let $(X,d)$ be a geodesic metric space, and let $x,y,z\in X$.
Then $(x,y)_z=0$ if and only if there exists a geodesic from $x$ to $y$ that
contains $z$.
\end{lemma}
\begin{proof}
As for \refL{L1}.
\end{proof}

In a geodesic metric space, it is enough to assume the hyperbolicity condition 
\eqref{4gda} when the \lhs{} is 0, \ie, by \refL{LH3}, when $w$ is on a
geodesic joining $x$ and $y$.

\begin{lemma}\label{LH4}
  Let $(X,d)$ be a geodesic metric space.
  Suppose that, for every $x,y,z\in X$ and every $w$ that lies on a geodesic
  from $x$ to $y$, the inequality \eqref{4gda} holds, or equivalently
\begin{align}\label{lh4}
 (x,z)_w\bmin(y,z)_w\le \gd.
\end{align}
Then $X$ is $3\gd$-hyperbolic.
\end{lemma}

\begin{proof}
  See \cite[p.~286 and Exercise 8.4.5]{Burago}
or \cite[Proof of Theorem 2.34]{Vaisala}.
\end{proof}

\begin{remark}
  In particular, using \refT{T4G},
a geodesic metric space is a real tree if and only if \eqref{lh4} holds with
$\gd=0$.
In fact, it is easily seen that this condition implies that geodesics are
unique, so \ref{T1} holds. (Given two geodesics from $x$ to $y$, let $z$ and
$w$ be points on them with $d(z,x)=d(w,x)$, and conclude $z=w$ from
\eqref{lh4}.) Then,  by \refL{LH3} again, \eqref{lh4} with $\gd=0$ 
when $w\in[x,y]$
is
equivalent to $[x,y]\subseteq[x,z]\cup[y,z]$, which is \ref{T2Y}. 
\end{remark}

The condition in \refL{LH4} can be regarded as a condition on the geometry
of the triangle $xyz$. One can pursue this point of view further.

\begin{definition}\label{D3}
  Let $(X,d)$ be a geodesic metric space.
A \emph{triangle} $xyz$
in $X$ is a
  set of three \emph{vertices}
$x,y,z\in X$ together with three \emph{sides} $xy$, $xz$ and $yz$ that are
geodesics between the pairs of vertices.
\end{definition}

Note that in general, the sides of a triangle are not uniquely determined by
the vertices.

\begin{definition}\label{Dslim}
  Let $(X,d)$ be a geodesic metric space.
  A triangle is \emph{$\gd$-slim} if each side  is contained in the closed
  $\gd$-neighbourhood of the union of the two other sides,
\ie, if every $w\in xy$ has distance $\le\gd$ to $xz\cup yz$, and similarly
for the two other sides.
\end{definition}

\citet{Gromov} used a definition that is similar, but somewhat more technical.
Note that in a triangle with vertices $x,y,z$, we have, by \eqref{ak4},
$(y,z)_x+(x,z)_y=\dxy$, and thus the side $xy$ can be divided into two parts
of lengths $(y,z)_x$ and $(x,z)_y$, containing $x$ and $y$ respectively; we
call these, the \emph{$x$-part} and \emph{$y$-part} of $xy$.
(In a tree, these are the parts of $[x,y]$ before and after $\gam(x,y,z)$.
The definition in \cite{Gromov} is stated in terms of isometric mappings of
the three sides into a real tree with three endpoints having the same
distances between each other as $x,y,z$.)

\begin{definition}\label{Dthin}
  Let $(X,d)$ be a geodesic metric space.
  A triangle $xyz$ is \emph{$\gd$-thin} if 
for each $w$ in the $x$-part of the side $xy$, the point $v\in xz$ with
$d(v,x)=d(w,x)$ satisfies $d(v,w)\le\gd$,
and similarly for the $y$-part and for the two other sides.
\end{definition}

It is immediate that a $\gd$-thin triangle is $\gd$-slim.

\begin{lemma}\label{Lthin}
Let $X$ be a geodesic metric space.
\begin{romenumerate}
\item \label{Lthina}
If $X$ is $\gd$-hyperbolic, then every triangle in $X$ is $2\gd$-thin.  
\item \label{Lthinb}
If every triangle in $X$ is $\gd$-thin, then $X$ is $2\gd$-hyperbolic.
\end{romenumerate}
\end{lemma}
\begin{proof}
  See \cite[Proposition 6.3.C]{Gromov}.
\end{proof}

\begin{lemma}\label{Lslim}
Let $X$ be a geodesic metric space.
\begin{romenumerate}
\item \label{Lslima}
If $X$ is $\gd$-hyperbolic, then every triangle in $X$ is $2\gd$-slim.  
\item \label{Lslimb}
If every triangle in $X$ is $\gd$-slim, then $X$ is $3\gd$-hyperbolic.
\end{romenumerate}
\end{lemma}
\begin{proof}
\ref{Lslima} follows from \refL{Lthin}\ref{Lthina}; see also 
 \cite[Theorems 2.35]{Vaisala}
(with $3\gd$).

\ref{Lslimb} is
 \cite[Theorems 2.34  (with $h=0$)]{Vaisala}.
\end{proof}

\begin{remark}\label{Rauk}
  If every triangle in $X$ is $\gd$-slim (or $\gd$-thin), then
any two geodesics with the same endpoints are within distance $2\gd$ of
each other. (I.e., their Hausdorff distance is $\le2\gd$.)
This follows by considering 
the triangle obtained by subdividing one of the geodesics at an arbitrary
interior point. (Or, more bravely, by considering the geodesics as two sides of
a degenerate triangle with two vertices coinciding.)
\end{remark}

\begin{remark}\label{Ris}
  In particular, using \refT{T4G}, a geodesic metric space is a real tree if
  and only if every triangle is 0-slim (or 0-thin).

In fact, if this holds, then by \refR{Rauk}, geodesics are unique, and thus
\ref{T1} holds. Then, a triangle $xyz $ is $0$-slim if and only if \ref{T2Y}
holds for all permutations of $x,y,z$.
Conversely, each triangle in a real tree is $0$-thin as a consequence of
\refLs{LgD} and \ref{Lddd}.
\end{remark}


\newcommand\AAP{\emph{Adv. Appl. Probab.} }
\newcommand\JAP{\emph{J. Appl. Probab.} }
\newcommand\JAMS{\emph{J. \AMS} }
\newcommand\MAMS{\emph{Memoirs \AMS} }
\newcommand\PAMS{\emph{Proc. \AMS} }
\newcommand\TAMS{\emph{Trans. \AMS} }
\newcommand\AnnMS{\emph{Ann. Math. Statist.} }
\newcommand\AnnPr{\emph{Ann. Probab.} }
\newcommand\CPC{\emph{Combin. Probab. Comput.} }
\newcommand\JMAA{\emph{J. Math. Anal. Appl.} }
\newcommand\RSA{\emph{Random Structures Algorithms} }
\newcommand\DMTCS{\jour{Discr. Math. Theor. Comput. Sci.} }

\newcommand\AMS{Amer. Math. Soc.}
\newcommand\Springer{Springer-Verlag}
\newcommand\Wiley{Wiley}

\newcommand\vol{\textbf}
\newcommand\jour{\emph}
\newcommand\book{\emph}
\newcommand\inbook{\emph}
\def\no#1#2,{\unskip#2, no. #1,} 
\newcommand\toappear{\unskip, to appear}

\newcommand\arxiv[1]{\texttt{arXiv:#1}}
\newcommand\arXiv{\arxiv}

\def\nobibitem#1\par{}

\end{document}